\documentclass[10pt]{article}
\usepackage{amsmath,amssymb,amsthm}
\usepackage{amsfonts, amsmath, amsthm, amssymb}
\usepackage{epsfig}
\usepackage{color}

\title{Magic squares with empty cells}
\author {
Abdollah Khodkar and David Leach\\
Department of Mathematics\\
University of West Georgia\\
Carrollton, GA 30118\\
{\tt akhodkar@westga.edu},
{\tt cleach@westga.edu}
}

\date{}

\newtheorem{prelem}{{\bf Theorem}}

 \newtheorem{theorem}{Theorem}

\newtheorem{lemma}[theorem]{Lemma}

\theoremstyle{definition}

\theoremstyle{remark}

\begin{document}

\maketitle

\begin{abstract}
\noindent A {\em $k$-magic square} of order $n$ is an arrangement of the numbers from 0 to $kn-1$ in an $n\times n$ matrix, such that each row and each column has exactly $k$ filled cells, each number occurs exactly once, and the sum of the entries of any row or any column is the same. A magic square is called $k$-{\em diagonal} if its entries all belong to $k$ consecutive diagonals.
In this paper we prove that a $k$-diagonal magic square exists if and only if $n=k=1$ or $3\leq k\leq n$ and
$n$ is odd or $k$ is even.

\end{abstract}

\section{Introduction}\label{SEC1}
A {\em magic square} of order $n$ is an arrangement of the numbers from 0 to $n^2-1$ in an $n\times n$ array such that each number occurs exactly once in the array and the sum of the entries of any row or any column is the same,  which is called the {\em magic sum}. It is well-known that (see \cite {sun}) a magic square exists for each $n\geq 1$.
A {\em $k$-magic square} of order $n$ is an arrangement of the numbers from 0 to $kn-1$ in an $n\times n$ array such that each row and each column has exactly $k$ filled cells, each number occurs exactly once in the array, and the sum of the entries of any row or any column is the same.

An {\em integer Heffter array} $H(m, n; s, t)$ is an $m\times n$ array with entries from
$X=\{\pm1,\pm2,$ $\ldots,\pm ms\}$
such that each row contains $s$ filled cells and each column contains $t$ filled cells,
the elements in every row and column sum to 0 in ${\mathbb Z}$, and
for every $x\in X$, either $x$ or $-x$ appears in the array.
The notion of an integer Heffter array $H(m, n; s, t)$ was first defined in  \cite{arc1}.
Integer Heffter arrays with $m=n$ represent a type of magic square where each number from the set
$\{1,2,\ldots,n^2\}$ is used once up to sign.
A {\em square} integer Heffter array $H(n; k)$ is an integer Heffter array with $m=n$ and there are
exactly $k$ filled cells in each row and each column.
In \cite{ADDY,DW} it is proved that:

\begin{theorem}\label{Heffterwithemptycells}
There is an integer $H(n; k)$ if and only if $3\leq k\leq n$ and $nk\equiv 0,3 \pmod 4$.
\end{theorem}

A {\em signed magic array} $SMA(m,n;s,t)$ is an $m \times n$ array with entries from $X$, where
$X=\{0,\pm1,\pm2,\ldots,\pm (ms-1)/2\}$ if $ms$ is odd and $X = \{\pm1,\pm2,\ldots,\pm ms/2\}$ if $ms$ is even,
such that precisely $s$ cells in every row and $t$ cells in every column are filled,
every integer from set $X$ appears exactly once in the array, and
the sum of each row and of each column is zero. The signed magic squares also represent a type of magic square where each number from the set $X$ is used once.
The notion of a signed magic array $SMA(m,n;s,t)$  was first defined in \cite{KSW}.
In the case where $m = n$, the array is called {\em signed magic square}.
The notation $SMS(n;k)$ is used for a signed magic square with $k$ filled cells in each row and
in each column. In \cite{KSW} it is proved that:

\begin{theorem} \label{signedsquare}
There is an $SMS(n;k)$ precisely when $n=k =1$ or $3\leq k\leq n$.
\end{theorem}

A magic square is called $k$-{\em diagonal} if its entries all belong to $k$ consecutive diagonals (this includes broken diagonals as well). In this paper, similar to Theorems \ref {Heffterwithemptycells} and
 \ref {signedsquare}, we prove that a $k$-diagonal magic square exists if and only if
$n=k=1$ or $3\leq k\leq n$ and $n$ is odd or $k$ is even.

Often the definition of a magic square requires the additional condition that the sum of the diagonals are also constant and equal to the sum of the rows and columns. Due to the structure of a $k$-diagonal magic square as given in this paper, this requirement is not well defined, as the number of entries is either $n$ or $k$. Thus we do not require this condition.

\section{Direct constructions}\label{SEC2}
In this section we present direct constructions for $k$-diagonal
magic squares of order $n$ for $k=3,4,5$ and $6$.
It is easy to see that if there exists a magic square of order $n$ with precisely $k$ filled cells in
each row and each column, then the magic sum must be $k(kn-1)/2$. Hence, if $k$ is odd, then $n$ must be odd.
We use the notation $(i,j;e)$ for an element of an array $M$ to indicate that the entry $e$ is in row $i$ and column $j$ of $M$.

\begin{theorem}\label{ThreefilledCells}
{\rm There exists a $3$-diagonal magic square of order $n\geq 3$ and $n$ is odd.}
\end{theorem}

\begin{proof}
We fill three consecutive diagonals of an empty array $M$ of order $n$ as follows:

\noindent Diagonal 1: $ \left\{\begin{array}{lll}
(2i,n+2i-3;(n-2i-1)/2)& \mbox {for}& 0\leq i\leq (n-1)/2,\\
(2i+1,n+2i-2; n-i-1) & \mbox {for}& 0\leq i\leq(n-3)/2.\\
\end{array}\right.$

\noindent Diagonal 2: $\begin{array}{lll}
(i,n+i-2;2n+i)& \mbox {for}& 0\leq i\leq n-1.\\
\end{array}$

\noindent Diagonal 3: $\left\{ \begin{array}{lll}
(2i,n+2i-1;2n-i-1)& \mbox {for}& 0\leq i\leq (n-1)/2,\\
(2i+1,2i; (3n-2i-3)/2 & \mbox {for}& 0\leq i\leq(n-3)/2.\\
\end{array}\right.$

\noindent
Addition in the first and second components of each triple $(i,j;e)$ is done modulo $n$.
(See Figure \ref{9.3} for a 3-diagonal magic square of order 9.)
We now prove that $M$ is an $(n;3)$ magic square. First
note that the numbers in Diagonal 1 are $0, 1, 2, \ldots, n-1$,
in Diagonal 2 are $2n, 2n+1, \ldots, 3n-1$, and in Diagonal 3 are $n, n+1,\ldots, 2n-1$.

Second, we calculate the row sums and column sums.
In row $2i$, where $0\leq i\leq(n-1)/2$, we have $(n-2i-1)/2+(2n+2i)+(2n-i-1) =3(3n-1)/2$ and in row $2i+1$, where $0\leq i\leq (n-3)/2$, we have
$(n-i-1)+(2n+2i+1)+(3n-2i-3)/2 =3(3n-1)/2$.

The column sum for column $2j$, $0\leq j\leq (n-3)/2$, is
$((3n-2j-3)/2)+(2n+2j+2)+(n-j-2)=3(3n-1)/2$. The column sum for column
$2j+1$, $0\leq j\leq (n-5)/2$, is $(2n-j-2)+(2n+2j+3)+((n-2j-5)/2)=3(3n-1)/2.$

Finally, the column sum for column $n-2$ is $((3n-1)/2)+(n-1)+(2n)=3(3n-1)/2$ and
for column $n-1$ is $((n-3)/2)+(2n+1)+(2n-1)=3(3n-1)/2.$
 This completes the proof.
\end{proof}

\begin{figure}[ht]
$$\begin{array}{|c|c|c|c|c|c|c|c|c|}\hline
 & & & & & &4&18&17 \\ \hline
12& & & & & & &8&19 \\ \hline
20&16& & & & & & &3 \\ \hline
7&21&11& & & & & &  \\ \hline
 &2&22&15& & & & &  \\ \hline
 & &6&23&10& & & &  \\ \hline
 & & &1&24&14& & &  \\ \hline
 & & & &5&25&9& &  \\ \hline
 & & & & &0&26&13&  \\ \hline
 \end{array}$$
 \caption{A 3-diagonal magic square of order 9}
 \label{9.3}
\end{figure}

\begin{theorem}\label{FourfilledCells}
{\rm There exists a $4$-diagonal magic square of order $n\geq 4$.}
\end{theorem}

\begin{proof}
We fill four consecutive diagonals of an empty array $M$ of order $n$ as follows:

\noindent Diagonal 1: $ \begin{array}{lll}
(i,i+1;i) & \mbox{for} & 0\leq i\leq n-1.
\end{array}$

\noindent Diagonal 2: $ \begin{array}{lll}
(i,i+2;4n-i-1) & \mbox{for} & 0\leq i\leq n-1.
\end{array}$

\noindent Diagonal 3: $ \begin{array}{lll}
(n+i-2,i+1;3n-i-1) & \mbox{for} & 0\leq i\leq n-1.
\end{array}$

\noindent Diagonal 4: $ \begin{array}{lll}
(n+i-2,i+2;n+i) & \mbox{for} & 0\leq i\leq n-1.
\end{array}$

\noindent
Addition in first and second components of each triple $(i,j;e)$ is done modulo $n$.
(See Figure \ref{9.4} for a 4-diagonal magic square of order 9.)
We now prove the resulting square is magic. First
note that the numbers in Diagonal 1 are $0, 1, 2, \ldots, n-1$,
in Diagonal 2 are $3n, 3n+1, \ldots, 4n-1$, in Diagonal 3 are $2n, 2n+1,\ldots, 3n-1$ and in diagonal 4 are
$n, n+1, \ldots, 2n-1$.

Second, we calculate the row sums and column sums. In row $0\leq i\leq n-3$ we have
$i+(4n-i-1)+(3n-i-3)+(n+i+2)=8n-2$. In column $2\leq j\leq n-1$ we have $(j-1)+(4n-j+1)+(3n-j)+(n+j-2)=8n-2$.
Row sum for row $n-1$ is $(n-1)+(3n)+(3n-2)+(n+1)=8n-2$ and for row $n-2$ is $(n-2)+(3n+1)+(3n-1)+(n)=8n-2$.
Column sum for column zero is $(n-1)+(3n+1)+(2n)+(2n-2)=8n-2$ and for column 1 is
$0+(3n)+(3n-1)+(2n-1)=8n-2$.
This completes the proof.
\end{proof}

 \begin{figure}[ht]
 $$\begin{array}{|c|c|c|c|c|c|c|c|c|}\hline
  &0&35&24&11& & & & \\ \hline
 & &1&34&23&12& & & \\ \hline
 & & &2&33&22&13& & \\ \hline
 & & & &3&32&21&14& \\ \hline
 & & & & &4&31&20&15\\ \hline
16& & & & & &5&30&19\\ \hline
18&17& & & & & &6&29\\ \hline
28&26&9& & & & & &7\\ \hline
8&27&25&10& & & & & \\ \hline
\end{array}\\$$
 \caption{A 4-diagonal magic square of order 9}
 \label{9.4}
\end{figure}

\begin{theorem}\label{FivefilledCells}
{\rm There exists a $5$-diagonal magic square of order $n\geq 5$, where $n$ is odd.}
\end{theorem}

\begin{proof}
We fill five consecutive diagonals of an empty array $M$ of order $n$ as follows:

\noindent Diagonal 1: $ \begin{array}{lll}
(i,i+1;i) & \mbox{for} & 0\leq i\leq n-1.
\end{array}$

\noindent Diagonal 2: $\left\{ \begin{array}{lll}
(n+2i-1,2i+1;2n-i-1) & \mbox{for} & 0\leq i\leq (n-1)/2\\
(2i,2i+2;(3n-3)/2-i) & \mbox{for} & 0\leq i\leq (n-3)/2.\\
\end{array}\right.$

\noindent Diagonal 3: $\begin{array}{lll}
(i,i+3;3n-i-1) & \mbox{for} & 0\leq i\leq n-1.
\end{array}$

\noindent Diagonal 4: $\left\{\begin{array}{lll}
(n+2i-2,2i+2;4n-i-1) & \mbox{for} & 0\leq i\leq (n-1)/2\\
(n+2i-1,2i+3;(7n-3)/2-i) & \mbox{for} & 0\leq i\leq (n-3)/2.\\
\end{array}\right.$

\noindent Diagonal 5: $\begin{array}{lll}
(n+i-2,i+3;4n+i) & \mbox{for} & 0\leq i\leq n-1.
\end{array}$

\noindent
Addition in first and second components of each triple $(i,j;e)$ is modulo $n$.
(See Figure \ref{11.5} for a 5-diagonal magic square of order 9.)

Note that the entries in diagonal $d$, $1\leq d\leq 5$, are $\{n(d-1),n(d-1)+1,\ldots,nd-1\}$.

The row sum for row $2i$, where $0\leq i \leq (n-3)/2$, is
$$2i+(3n-3)/2-i+3n-2i-1+4n-i-2+4n+2i+2=5(5n-1)/2.$$

\noindent The row sum for row $2i+1$, where $0\leq i \leq (n-5)/2$, is
$$2i+1+2n-i-2+3n-2i-2+(7n-3)/2-i-1+4n+2i+3=5(5n-1)/2.$$
Finally, the row sum for row $n-1$ is $n-1+2n-1+2n+(7n-3)/2+4n+1=5(5n-1)/2$ and for row $n-2$ is $n-2+(3n-1)/2+2n+1+4n-1+4n=5(5n-1)/2$.

We now calculate the column sums.
The column sum for column $2j$, where $1\leq j\leq (n-1)/2$, is
$$2j-1+(3n-3)/2-j+1+3n-2j+2+4n-j+4n+2j-3=5(5n-1)/2.$$
The column sum for column $2j+1$, where $1\leq j\leq (n-3)/2$, is
$$2j+2n-j-1+3n-2j+1+(7n-3)/2-j+1+4n+2j-2=5(5n-1)/2.$$

The sum for column zero is $n-1+(3n-1)/2+2n+2+3n+5n-3=5(5n-1)/2$,
for column 1 is $0+2n-1+2n+1+(7n-1)/2+5n-2=5(5n-1)/2$ and
for column 2 is $1+(3n-3)/2+2n+4n-1+5n-1=5(5n-1)/2.$
This completes the proof.
\end{proof}

\vspace{15mm}

\begin{figure}[ht]
$$\begin{array}{|c|c|c|c|c|c|c|c|c|c|c|}
\hline
	&	0	&	15	&	32	&	42	&	46	&		&		&		&		&		\\ \hline
	&		&	1	&	20	&	31	&	36	&	47	&		&		&		&		\\ \hline
	&		&		&	2	&	14	&	30	&	41	&	48	&		&		&		\\ \hline
	&		&		&		&	3	&	19	&	29	&	35	&	49	&		&		\\ \hline
	&		&		&		&		&	4	&	13	&	28	&	40	&	50	&		\\ \hline
	&		&		&		&		&		&	5	&	18	&	27	&	34	&	51	\\ \hline
52	&		&		&		&		&		&		&	6	&	12	&	26	&	39	\\ \hline
33	&	53	&		&		&		&		&		&		&	7	&	17	&	25	\\ \hline
24	&	38	&	54	&		&		&		&		&		&		&	8	&	11	\\ \hline
16	&	23	&	43	&	44	&		&		&		&		&		&		&	9	\\ \hline
10	&	21	&	22	&	37	&	45	&		&		&		&		&		&		\\ \hline

\end{array}$$
\caption{A 5-diagonal magic square of order 11}
\label{11.5}
\end{figure}

\begin{theorem}\label{SixfilledCells}
{\rm There exists a $6$-diagonal magic square of order $n\geq 6$.
}
\end{theorem}

\begin{proof}
Let $n\geq 6$.
We fill six consecutive diagonals of an empty array $M$ of order $n$ as follows:

\noindent Diagonal 1: $ \begin{array}{lll}
(i,i+1;i)& \mbox{for} & 0\leq i\leq n-1.
\end{array}$

\noindent Diagonal 2: $ \begin{array}{lll}
(n-1+i,i+1;2n-1-i)& \mbox{for} & 0\leq i\leq n-1.
\end{array}$

\noindent Diagonal 3: $ \begin{array}{lll}
(i,i+3;2n+i)& \mbox{for} & 0\leq i\leq n-1.
\end{array}$

\noindent Diagonal 4: $ \begin{array}{lll}
(n-1+i,i+3;4n-1-i)& \mbox{for} & 0\leq i\leq n-1.
\end{array}$

\noindent Diagonal 5: $ \begin{array}{lll}
(i,i+5,;6n-2-2i)& \mbox{for} & 0\leq i\leq n-1.
\end{array}$

\noindent Diagonal 6: $ \begin{array}{lll}
(n-1+i,i+5;4n+1+2i)& \mbox{for} & 0\leq i\leq n-1.
\end{array}$

\noindent
Addition in first and second components of each triple $(i,j;e)$ is modulo $n$.
(See Figure \ref{10.6} shows a 6-diagonal magic square of order 10.)

Note that the entries in diagonal $d$, $1\leq d\leq 4$, are $\{n(d-1),n(d-1)+1,\ldots,nd-1\}$.
The entries in diagonal 5 are $\{4n,4n+2,4n+4,\ldots,6n-2\}$ and the entries in diagonal 6 are
$\{4n+1,4n+3,4n+5,\ldots,6n-1\}$.

The row sum for row $n-1$ is $n-1+2n-1+3n-1+4n-1+4n+4n+1=18n-3$ and for row $i$, where $0\leq i \leq n-2$, is
$$(i)+(2n-i-2)+(2n+i)+(4n-i-2)+(6n-2i-2)+(4n+2i+3)=18n-3.$$

The column sum for column $j$, where $5\leq j\leq n-1$, is
$$(6n-2j+8)+(4n-j+2)+(2n+j-3)+(2n-j)+(j-1)+(4n+2j-9)=18n-3.$$
It is straightforward to see that the column sum for column $0\leq j\leq 4$ is also $18n-3$.
This completes the proof.

\end{proof}

\begin{figure}[ht]
$$\begin{array}{|c|c|c|c|c|c|c|c|c|c|}\hline
&0&18&20&38&58&43&&&\\ \hline
&&1&17&21&37&56&45&&\\ \hline
&&&2&16&22&36&54&47&\\ \hline
&&&&3&15&23&35&52&49\\ \hline
51&&&&&4&14&24&34&50\\ \hline
48&53&&&&&5&13&25&33\\ \hline
32&46&55&&&&&6&12&26\\ \hline
27&31&44&57&&&&&7&11\\ \hline
10&28&30&42&59&&&&&8\\ \hline
9&19&29&39&40&41&&&&\\ \hline
\end{array}$$
\caption{A 6-diagonal magic square of order 10}
\label{10.6}
\end{figure}

\section{Main Theorem}\label{SEC3}

The following two lemmas are crucial for the construction of a $k$-diagonal magic square of order $n$.

\begin{lemma}\label{shiftlemma}
{\rm Let $M$ be a $k$-diagonal magic square of order $n$.
\begin{itemize}
\item [(i)] Let $2k\leq n$. If we shift every $(i,j;e)\in M$ to $(i+k,j;e)$ (or to $(i,j+k;e)$),
addition is modulo $n$, the resulting square is a $k$-diagonal magic square of order $n$.

\item [(ii)] If we add $m$ to each nonempty cell of $M$, the resulting square has row sum and column sum
$km+k(nk-1)/2$ and its entries are $m,m+1,\ldots,mk-1+m.$
\end{itemize}
}
\end{lemma}

\begin{proof}
For Part (i) note that the entries in each row or each column will be the same after the shift. So the resulting array is a $k$-diagonal magic square of order $n$. The proof of Part (ii) is trivial.
\end{proof}

\begin{lemma}\label{ellplusklemma}
{\rm If there exist $\ell$-diagonal and  $m$-diagonal magic squares of order $n$ and $\ell+m\leq n$,
then there exists an $(\ell+m)$-diagonal magic square of order $n$.
}
\end{lemma}

\begin{proof}
Let $A$ and $B$ be $\ell$-diagonal and $m$-diagonal magic squares of order $n$, respectively.
Since $\ell+m\leq n$, without loss of generality, by Part (i) of Lemma \ref{shiftlemma}, we can assume if a cell $(i,j)$ of $A$ is filled, the cell $(i,j)$ of $B$ is empty. In addition, if we superimpose $A$ and $B$, we will have
$\ell+m$ consecutive diagonals.
Now we form a magic square $C$ of order $n$ with $\ell+m$ entries in each row and in each column as follows: If $(i,j;e)\in A$, then $(i,j;e)\in C$ and if $(i,j;e)\in B$, then $(i,j;e+\ell n)\in C$. Then $C$ is an $(\ell+m)$-diagonal magic square of order $n$ by Part (ii) of Lemma \ref{shiftlemma}.
\end{proof}

Figure \ref{9.7} shows a 7-diagonal magic square of order 9 obtained by the construction given in
Lemma \ref{ellplusklemma} and the 3-diagonal and 4-diagonal magic squares given in Figures \ref{9.3}
and \ref{9.4}, respectively.

\begin{figure}[ht]
$$\begin{array}{|c|c|c|c|c|c|c|c|c|c|}\hline
27&	62&	51&	38&	  &	  &	4 &	18&	17 \\ \hline
12&	28&	61&	50&	39&	  &	  &	8 &	19 \\ \hline
20&	16&	29&	60&	49&	40&	  &	  &	3   \\ \hline
7 &	21&	11&	30&	59&	48&	41&	  &	   \\ \hline
  &	2 &	22&	15&	31&	58&	47&	42&	   \\ \hline
  &	  &	6 &	23&	10&	32&	57&	46&	43  \\ \hline
44&	  &   	  &	1 &	24&	14&	33&	56&	45  \\ \hline
53&	36&	  &	  &	5 &	25&	9 &	34&	55  \\ \hline
54&	52&	37&	  &	  &	0 &	26&	13&	35  \\ \hline
\end{array}$$
\caption{A 7-diagonal magic square of order 9}
\label{9.7}
\end{figure}

\begin{lemma}\label{kdiagonal.n.odd}
{\rm
Let $n$ be odd and $3\leq k\leq n$. Then there exists a $k$-diagonal magic square of order $n$.
}
\end{lemma}

\begin{proof}
It is easy to see that there are nonnegative integers $a,b,c$ such $k=3a+4b+5c$. By Theorems
\ref{ThreefilledCells}, \ref{FourfilledCells} and \ref{FivefilledCells}, there are $3$-diagonal, $4$-diagonal and $5$-diagonal magic squares of order $n$. Apply Lemma \ref{ellplusklemma} to obtain a $k$-diagonal magic square of order $n$.
\end{proof}

\begin{lemma}\label{kdiagonal.n.even}
{\rm
Let $n,k$ be even and $4\leq k\leq n$. Then there exists a $k$-diagonal magic square of order $n$.
}
\end{lemma}

\begin{proof}
It is easy to see that there are nonnegative integers $a$ and $b$ such that $k=4a+6b$. By Theorems
\ref{FourfilledCells} and \ref{SixfilledCells}, there exist $4$-diagonal and $6$-diagonal magic squares of order $n$. Apply Lemma \ref{ellplusklemma} to obtain a $k$-diagonal magic square of order $n$.
\end{proof}

We are now ready to state the main theorem of this paper. Recall that the magic sum of an $(n;k)$ magic square
is $k(nk-1)/2$. Hence, if $n$ is even and $k$ is odd, there is no $(n;k)$ magic square.
\vspace{4mm}

\noindent {\bf Main Theorem}\quad
There exists a $k$-diagonal magic square of order $n$ if and only if $n=k=1$ or
$3\leq k\leq n$ and $n$ is odd or $k$ is even.


\end{document}